\documentclass[12pt,reqno]{amsart}
\usepackage{amssymb, amsmath,amsthm}
\usepackage{color}
\usepackage{cite}
\newtheorem{thm}{Theorem}

\newtheorem{Lemma}{Lemma}

\numberwithin{defn}{section}
\numberwithin{thm}{section}
\numberwithin{Lemma}{section}
\numberwithin{Corollary}{section}
\numberwithin{Example}{section}
\numberwithin{subsection}{section}
\numberwithin{Remark}{section}
\numberwithin{equation}{section}
\numberwithin{ppn}{section}
\begin{document}
\title
[Broadening the convergence domain of seventh order...]
{Broadening the convergence domain of Seventh-order method satisfying Lipschitz and H\"{o}lder conditions } 
\author{Akanksha Saxena, J. P. Jaiswal, K. R. Pardasani}
\date{}
\maketitle

\textbf{Abstract.} 
In this paper, the local convergence analysis of the multi-step seventh order method is presented for solving nonlinear equations assuming that the first-order  Fr\'echet derivative belongs to the Lipschitz class. The significance of our work is that it avoids the standard practice of Taylor expansion thereby, extends the applicability of the scheme by applying the technique based on the first-order derivative only. Also, this study provides radii of balls of convergence, the error bounds in terms of distances in addition to the uniqueness of the solution. Furthermore, generalization of this analysis satisfying H\"{o}lder continuity condition is provided since it is more relaxed than Lipschitz continuity condition. We have considered some numerical examples and computed the radii of the convergence balls.  
\\\\
\textbf{Mathematics Subject Classification (MSC2020).} 
65H10, 65J15, 65G99, 47J25.
%65H10, 65J15, 47J25,49M15,  65G99.

						%Introduction

\section{\bf Introduction}
One of the most crucial problems of numerical analysis concerns with finding efficiently and precisely the approximate locally unique solution $x^*$ of the nonlinear equation having the form of
\begin{equation}\label{eqn:11}
T(x)=0,
\end{equation}
where $T$ is defined on a convex open subset $D$ of a Banach space $X$ with values in a Banach space $Y$. Many problems in various fields of computational science, engineering and other disciplines involves optimization, differential equations, integral equations, radiative transfer theory, can be brought in a form like $(\ref{eqn:11})$ using mathematical modeling. Analytical methods of solving such type of problems are very scarce or almost non existent. Therefore, many researchers only rely on iterative methods and they have proposed a plethora of iterative methods.

The practice of numerical functional analysis for finding such solutions are widely and substantially connected to Newton-like methods which is defined as follows:
\begin{equation}
x_{n+1}=x_n-[T'(x_n)]^{-1}T(x_n)\ \ \ n\ge 0,
\end{equation}
is frequently used by various researchers as it has quadratic convergence (can be see in the ref.  \cite{Traub}). The other properties of Newton's method are established in the article \cite{Kantorovich}. Moreover, in some applications involving stiff systems, high-order methods are useful. Therefore, it is important to study high-order methods.
 
The local convergence analysis of iterative procedures is based on the information around a solution, to find estimates of the radii of the convergence balls. There exists many studies which deals with the local and the semilocal convergence analysis of Newton-like methods. In the last decades, many iterative methods of increasing order of convergence have been developed and have shown their efficiency, in numerical terms, like third order  [\cite{Argyros2}, \cite{Argyros1}], fourth order \cite{Argyros3} and fifth order [\cite{Cordero}, \cite{Gupta}] etc.

 Another issue of equal importance is obtaining the radius of convergence ball as well as developing a theory to extend the convergence domain. 
%Because, if we talk in the numerical language, the convergence domain is essential for the stable behaviour of an iterative scheme. 
Numerous authors have studied the local convergence analysis using Taylor's series but don't obtain the radii of convergence ball for the solution which can be seen in the refs. [\cite{Behl},\cite{Sharma}]. In terms of computing cost, it has a number of drawbacks. 
%This approach has been extended for the iterative method in Banach spaces for obtaining better theoretical results without following Taylor's series approach.It is no longer necessary to use higher-order derivatives to demonstrate the scheme's convergence. 
These types of techniques are discussed by many authors, for example, one can see the refs. [\cite{Traub}, \cite{Rall}]. %Analyzing an iterative algorithm's convergence ball can be useful for a variety of reasons including finding the radii of convergence balls, bounds on error $\|x_n-x^*\|$, and the region of uniqueness for the solution $x^*$. 
Using considered approach, it is possible to compute the convergence radii and the error  $\|x_n-x^*\|$ estimates. 
%Our research also provides a precise location for the solution $x^*$.

However, lower-order classical methods have been usually re-examined for only trying to increase its order of convergence.
In particular, Sharma and Gupta  \cite{Janak} constructed three steps method of order five, defined as follows:
{\small
\begin{eqnarray}\label{eqn:12}
y_n&=&x_n-\frac{1}{2}\Gamma_nT(x_n),\nonumber\\
z_n&=&x_n-[T'(y_n)]^{-1}T(x_n),\nonumber\\
x_{n+1}&=&z_n-[2[T'(y_n)]^{-1}-\Gamma_n]T(z_n),
\end{eqnarray}}
where, $\Gamma_n=[T'(x_n)]^{-1}$.
The local convergence of above multi-step Homeier's-like method has been studied by Panday and Jaiswal \cite{Bhavna} with the help of Lipschitz and H$\ddot{o}$lder continuity conditions. 
In the extension of three-step Homeier's method, Xiao and Yin  \cite{Xiao} developed a fourth-step seventh order convergent method, which is given as:
{\small
\begin{eqnarray}\label{eqn:13}
y_n&=&x_n-\frac{1}{2}\Gamma_nT(x_n),\nonumber\\
z_n^{(1)}&=&x_n-[T'(y_n)]^{-1}T(x_n),\nonumber\\
z_{n}^{(2)}&=&z_n^{(1)}-[2[T'(y_n)]^{-1}-\Gamma_n]T(z_n^{(1)}),\nonumber\\
x_{n+1}&=&z_n^{(2)}-[2[T'(y_n)]^{-1}-\Gamma_n]T(z_n^{(2)}).
\end{eqnarray}}
% where, $\Gamma_n=[T'(x_n)]^{-1}$.
This method requires the evaluation of three function, two first order derivatives and two matrix inversions per iteration.
In this article, we will present the answers with the range of initial guess $x^*$ that tell us how close the initial guess should be required for granted convergence of the method $(\ref{eqn:13})$.

In finding out the existence of the solution, we require the computation of the first-order Fr\'echet derivative. However, the analysis of convergence is established by applying the Taylor series approach based on higher-order derivatives. These techniques, which need the higher-order derivatives, limit the applicability of the algorithm. As an motivational illustration\cite{Parhi}, define a function $f$ on $D=[-\frac{1}{2},\frac{5}{2}]$ by
{\small
\begin{eqnarray}\label{eqn:14}
f(x)=
\begin{cases}
      x^3log(x^2)+x^5-x^4, & \text{if \ $x\neq0$}\\
      0, & \text{if \ $x=0$}. 
\end{cases}
\end{eqnarray}}
It is clearly notable that $f'''$ is unbounded on $D$. Hence, the theory based on higher-order derivatives \cite{Noor} fail to solve the above problem. Also, one get no idea about the domain of convergence \cite{Parhi}. The local convergence study gives valuable information regarding the radius of convergence ball.

In this presented paper, motivated by the foregoing observations and ongoing work in this area, we discuss the local convergence for the method $(\ref{eqn:13})$, by following the approach based on $f'$ to stay away from the evaluation of higher-order  Fr\'echet derivatives. Thereby, we enlarge the utility of the method $(\ref{eqn:13})$ by using hypotheses only on the first-order derivative of the function $T$ and Lipschitz and H\"{o}lder continuity conditions for which earlier studies can not be used due to the computation of higher-order Fr\'echet derivatives.

The arrangement of the whole text is as follows: the "Local convergence analysis of the method $(\ref{eqn:13})$ under Lipschitz condition" section deals with the local convergence results of the method $(\ref{eqn:13})$. Similarly, the "Local convergence analysis of the method $(\ref{eqn:13})$ under H\"{o}lder condition" section deals with the local convergence results of the method $(\ref{eqn:13})$ followed by the "Applications with Numerical Examples" section. The "Conclusions" section is placed in the last section. 
								%Preliminary Results%

\section{\bf Local convergence analysis of the method $(\ref{eqn:13})$ under Lipschitz condition}
In this section, we describe the the local convergence analysis of the method $(\ref{eqn:13})$ which is centered on some parameters and scalar functions. Considering $\psi_0>0$ and $\psi>0$ be two parameters with $\psi_0\leq \psi$, we define the functions  $\eta_1,\eta_2,\eta_3,\eta_4,p,H_1,H_2,H_3$ and $H_4$ on interval $[0,\frac{1}{\psi_0})$ by
{\small
\begin{eqnarray}\label{eqn:32}
\eta_4(a)&=&{\bigg[}\frac{\psi\eta_3(a)a}{2(1-\psi_0\eta_3(a)a)}\nonumber\\
&&+\frac{1}{1-\psi_0\eta_3(a)a}\psi[\eta_1(a)+\eta_3(a)]\times
\left(\frac{1}{1-p(a)}\bigg[1+\frac{\psi_0}{2}\eta_3(a)a\bigg]a\right)\nonumber\\
&&+\left(\frac{1}{1-\psi_0a}.\frac{\psi[a+\eta_1(a)a]}{1-p(a)}\times \bigg[1+\frac{\psi_0}{2}\eta_3(a)a\bigg]\right){\bigg]}\eta_3(a),
\end{eqnarray}}
where
\begin{eqnarray}\label{eqn:33}
\eta_1(a)&=&\frac{1}{1-\psi_0a}{\bigg[}\frac{\psi a}{2}+\frac{1+(\psi_0/2)a}{2}{\bigg]},
\end{eqnarray}
\begin{eqnarray}\label{eqn:34}
\eta_2(a)&=&\frac{1}{1-\psi_0a}\bigg[\frac{\psi a}{2}+\frac{\psi[1+\eta_1(a)][\psi_0/2a+1]a}{1-p(a)}\bigg],
\end{eqnarray}
{\small
\begin{eqnarray}\label{eqn:32a}
\eta_3(a)&=&{\bigg[}\frac{\psi\eta_2(a)a}{2(1-\psi_0\eta_2(a)a)}\nonumber\\
&&+\frac{1}{1-\psi_0\eta_2(a)a}\psi[\eta_1(a)+\eta_2(a)]\times
\left(\frac{1}{1-p(a)}\bigg[1+\frac{\psi_0}{2}\eta_2(a)a\bigg]a\right)\nonumber\\
&&+\left(\frac{1}{1-\psi_0a}.\frac{\psi[a+\eta_1(a)a]}{1-p(a)}\times \bigg[1+\frac{\psi_0}{2}\eta_2(a)a\bigg]\right){\bigg]}\eta_2(a),
\end{eqnarray}}
and
\begin{eqnarray}\label{eqn:35}
p(a)=\psi_0\eta_1(a)a.
\end{eqnarray}
Let
{\small
\begin{eqnarray}\label{eqn:36}
H_1(a)=\eta_1(a)-1,\ \ \ H_2(a)=\eta_2(a)-1,\\
H_3(a)=\eta_3(a)-1,\ \ \ H_4(a)=\eta_4(a)-1, 
\end{eqnarray}}
and 
\begin{equation}\label{eqn:31}
\rho_1=\frac{2}{2\psi+5\psi_0}<\frac{1}{\psi_0}.
\end{equation}
We observe that $\eta_1(\rho_1)=1$ and can attain 
\begin{equation}\label{eqn:31a}
0\leq \eta_1(a)<1\ for\ a\in [0,\rho_1),\ \ \ 0<\rho_1<1/\psi_0.
\end{equation}
Now, we have that $H_1(0)=H_2(0)=H_3(0)=H_4(0)<0$ and $H_1(a)\rightarrow+\infty$ or a positive constant as $t\rightarrow \frac{1}{\psi_0}^-$.The intermediate value theorem confirms the existence of the zeros of the function $H_2(a)$ in the interval $(0,\frac{1}{\psi_0})$. We denote the smallest zero of $H_2(a)$ as $\rho_2$. Also, $h_2(\rho_1)>0$ and $\rho_1<\frac{1}{\psi_0}$, which follows from that 
\begin{equation}\label{eqn:37}
0<\rho_2<\rho_1,\ \ \ 0<\eta_2(a)<1, for\ a\in [0,\rho_2).
\end{equation}
Following this procedure, there comes the existence of zeros of functions $H_i,\ i=1,2,3,4$ in the interval $(0,\rho_0)$. Denote by $\rho_i$, respectively, the smallest solution of functions $H_i's$. Define the radius of convergence $\rho$ and for each $a\in [0, \rho)$ by
\begin{eqnarray}\label{eqn:39}
\rho=min\{\rho_i\}, \ i=1,2,3,4 ;\ \ \ 0\leq \eta_i(a)<1.
\end{eqnarray}
%and
Let $B(x^*,\rho),\ \overline{B(x^*,\rho)}$ stand, respectively for the open and closed ball in $X$ such that $x^*\in X$ and of radius $\rho>0$. Next, we present the local convergence analysis of method $(\ref{eqn:13})$ using the preceding notations and Lipschitz conditions.

								%Theorem(1)
\begin{thm}\label{thm:31}
Suppose that $T:D\subset X\rightarrow Y$ be a continuously first order Fr\'echet differentiable operator. Presume $\psi_0>0$ and $\psi>0$ be given parameters. Assume that there exists $x^*\in D$ for all $x,y\in D$ and fulfill the below conditions:
\begin{eqnarray}\label{eqn:311}
T(x^*)=0,\ [T'(x^*)]^{-1}\in L(Y,X),
\end {eqnarray}
where, $L(X,Y)$ is the set of bounded linear operators from $X$ to $Y$,
\begin{eqnarray}\label{eqn:312}
\|[T'(x^*)]^{-1}(T'(x)-T'(x^*))\|&\le& \psi_0\|x-x^*\|.
\end{eqnarray}
\begin{eqnarray}\label{eqn:313}
\|[T'(x^*)]^{-1}(T'(x)-T'(y))\|&\le& \psi\|x-y\|,\\
B(x^*,\rho)&\subseteq& D,
\end{eqnarray}
%$\rho_1$ and $(\ref{eqn:31})$ and , respectively
where  $\rho$ is defined by equation  $(\ref{eqn:39})$. Then, the sequence $\{x_n\}$ generated by the method $(\ref{eqn:13})$ for $x_0\in B(x^*,\rho)\backslash\{x^*\}$ is well defined in $B(x^*,\rho)$ remains in $B(x^*,\rho)$ for each $n=0,1,2,\cdots$ and converges to $x^*$. Consequently, the following relations holds for $n=0,1,2,\cdots$:
\begin{eqnarray}\label{eqn:316}
\|y_n-x^*\|\le \eta_1(\|x_n-x^*\|)\|x_n-x^*\|\le\|x_n-x^*\|<\rho,
\end{eqnarray}
\begin{eqnarray}\label{eqn:317}
\|z_n^{(1)}-x^*\|\le \eta_2(\|x_n-x^*\|)\|x_n-x^*\|\le\|x_n-x^*\|<\rho,
\end{eqnarray}
\begin{eqnarray}\label{eqn:317a}
\|z_n^{(2)}-x^*\|\le \eta_3(\|x_n-x^*\|)\|x_n-x^*\|\le\|x_n-x^*\|<\rho,
\end{eqnarray}
and
\begin{eqnarray}\label{eqn:318}
\|x_{n+1}-x^*\|\le \eta_4(\|x_n-x^*\|)\|x_n-x^*\|\le\|x_n-x^*\|<\rho,
\end{eqnarray}
where the functions $\eta_i, i=1,2,3,4$ are defined by the expressions $(\ref{eqn:32})$ -$(\ref{eqn:32a})$. Moreover, if there exists $\varrho\in [\rho,\frac{1}{\psi_0})$ such that $\overline {B(x^*,\varrho)}\subseteq D$. Then the limit point $x^*$ is the only solution of equation $T(x)=0$ in $D_1=D\cap \overline {B(x^*,\varrho)}$.
\end{thm}

\begin{proof}
We shall show by mathematical induction that sequence $\{x_n\}$ is well defined and converges to $x^*$. Using the hypotheses, $x_0\in B(x^*,\rho)\backslash\{x^*\}$, equation $(\ref{eqn:31})$ and inequality $(\ref{eqn:312})$, we have 
\begin{eqnarray}
\|[T'(x^*)]^{-1}(T'(x_0)-T'(x^*))\|\le \psi_0\|x_0-x^*\|\le \psi_0(\rho)<1.
\end{eqnarray}
It follows from the above and the Banach lemma on invertible operator \cite{Argyros} that $[T'(x_0)]^{-1}\in L(Y,X)$
 or $T'(x_0)$ is invertible and
\begin{eqnarray}
\|[T'(x_0)]^{-1}T'(x^*)\|\le\frac{1}{1-\psi_0\|x_0-x^*\|}.
\end{eqnarray}
Now, $y_0$ is well defined by the first sub-step of the scheme $(\ref{eqn:13})$ and for $n=0,$
{\tiny
\begin{eqnarray}\label{eqn:321}
y_0-x^*&=&x_0-x^*-\frac{1}{2}[[T'(x_0)]^{-1}T(x_0)]\nonumber\\
& =&\frac{1}{2}[[T'(x_0)]^{-1}T(x_0)]+[T'(x_0)]^{-1}[T'(x_0)(x_0-x^*)-T(x_0)+T(x^*)]||.\nonumber\\
\end{eqnarray}}
Expanding $T(x_0)$ along $x^*$ and taking the norm, we get
{\tiny
\begin{eqnarray}\label{eqn:322a}
\|y_0-x^*\|&\le& \bigg\|\frac{1}{2}[T'(x_0)]^{-1}T'(x^*)\bigg\|\bigg[\int_0^1\|[T'(x^*]^{-1}[T'(x*+t(x_0-x^*)]dt\|\|x_0-x^*\|\bigg]\nonumber\\
&&+\|[T'(x_0)]^{-1}T'(x^*)\|\int_0^1\|[T'(x^*)]^{-1}[T'(x_0)-T'(x^*+t(x_0-x^*))]dt\|\|x_0-x^*\|.\nonumber\\
\end{eqnarray}}
Thus, we get
{\tiny
\begin{eqnarray}\label{eqn:323}
\|y_0-x^*\|&\le& \frac{1}{1-\psi_0\|x_0-x^*\|}\bigg[\frac{1}{2}{\bigg[}\frac{\psi_0}{2}\|x_0-x^*\|+1{\bigg]}\nonumber\\
&+&\frac{\psi}{2}\|x_0-x^*\|\bigg]\|x_0-x^*\|
\le\eta_1(\|x_0-x^*\|)\|x_0-x^*\|<\rho.
\end{eqnarray}}
From the inequalities $(\ref{eqn:312})$ and $(\ref{eqn:323})$, we have
{\small
\begin{eqnarray}
\|[T'(x^*)]^{-1}[T'(y_0)-T'(x^*)]\|&\le&\psi_0\|y_0-x^*\|\nonumber\\
&\le&\psi_0\eta_1(\|x_0-x^*\|)\|x_0-x^*\|\nonumber\\
&=&p(\|x_0-x^*\|)<1.
\end{eqnarray}}
Thus, by Banach lemma,
{\small
\begin{eqnarray}
\|[T'(y_0)]^{-1}T'(x^*)\|\le\frac{1}{1-p(\|x_0-x^*\|)}.
\end{eqnarray}}
From the second sub-step of the method $(\ref{eqn:13})$, we have
{\small
\begin{eqnarray}\label{eqn:326}
z_0^{(1)}-x^*&=&x_0-x^*-[T'(y_0)^{-1}T(x_0)]\nonumber\\
&=&x_0-x^*-[T'(x_0)]^{-1}T(x_0)+[T'(x_0)]^{-1}[T'(y_0)-T'(x_0)]T'(y_0)^{-1}T(x_0).\nonumber\\
\end{eqnarray}}
On taking norm of the equation $(\ref{eqn:326})$, we get
{\small
\begin{eqnarray}
\|z_0^{(1)}-x^*\|&\le&\|x_0-x^*-[T'(x_0)]^{-1}T(x_0)\|+\|[T'(x_0)]^{-1}T'(x^*)\|.\nonumber\\
&&\|[T'(x^*)]^{-1}[T'(y_0)-T'(x_0)]\|\|[T'(y_0)]^{-1}T'(x^*)\|\|[T'(x^*)]^{-1}T(x_0)\|\nonumber\\
&\le&\frac{1}{1-\psi_0\|x_0-x^*\|}\bigg[\frac{\psi}{2}\|x_0-x^*\|\nonumber\\
&+&\frac{[\psi\|y_0-x^*\|+\psi(\|x_0-x^*\|)][\frac{\psi_0}{2}\|x_0-x^*\|+1]}{1-p(\|x_0-x^*\|)}\bigg]\|x_0-x^*\|.
\end{eqnarray}}
Thus, we get
\begin{eqnarray}
\|z_0^{(1)}-x^*\|\le \eta_2(\|x_0-x^*\|)\|x_0-x^*\|\le\|x_0-x^*\|<\rho.
\end{eqnarray}
From the next sub-step of the method $(\ref{eqn:13})$, we have
{\small
\begin{eqnarray}\label{eqn:327}
&&z_0^{(2)}-x^*=z_0^{(1)}-x^*-(2[T'(y_0)]^{-1}-[T'(x_0)]^{-1})T(z_0^{(1)})\nonumber\\
&&=(z_0^{(1)}-x^*-[T'(z_0^{(1)})]^{-1}T'(z_0^{(1)}))+[T'(z_0^{(1)})]^{(-1)}T'(x^*)T'(x^*)^{-1}[T'(y_0)-T'(z_0^{(1)})]\nonumber\\
&&.[T'(y_0)]^{(-1)}T'(x^*)[T'(x^*)]^{-1}T(z_0^{(1)})+[T'(x_0)]^{(-1)}T'(x^*)[T'(x^*)]^{-1}[T'(y_0)-T'(x_0)]\nonumber\\
&&.[T'(y_0)]^{(-1)}T'(x^*)[T'(x^*)]^{-1}T(z_0^{(1)}).
\end{eqnarray}}
On expanding $T(z_0^{(1)})$ along $x^*$ and taking norm of the equation $(\ref{eqn:327})$, we get
{\tiny
\begin{eqnarray}
\|z_0^{(2)}-x^*\|&\le&\frac{1}{1-\psi_o\|z_0^{(1)}-x^*\|}\frac{\psi}{2}\|z_0^{(1)}-x^*\|.\|z_0^{(1)}-x^*\|+\frac{[\psi\|y_0-x^*\|+\psi\|z_0^{(1)}-x^*\|]}{1-\psi_0\|z_0^{(1)}-x^*\|}\nonumber\\
&\times&\|[T'(x^*)]^{-1}T(z_0^{(1)})\|+\frac{1}{1-\psi_0\|x_0-x^*\|}[\psi\|y_0-x^*\|+\psi\|x_0-x^*\|]\nonumber\\
&\times&\|[T'(x^*)]^{-1}T(z_0^{(1)})\|\nonumber\\
&\le&\frac{1}{1-\psi_0\|z_0^{(1)}-x^*\|}\frac{\psi}{2}\|z_0^{(1)}-x^*\|.\|z_0^{(1)}-x^*\|+\frac{[\psi\|y_0-x^*\|+\psi\|z_0^{(1)}-x^*\|]}{1-\psi_0\|z_0^{(1)}-x^*\|}\nonumber\\
&\times&\frac{1}{1-p(\|x_0-x^*\|)}\bigg(1+\frac{\psi_0}{2}\|z_0^{(1)}-x^*\|\bigg)\|z_0^{(1)}-x^*\|+\frac{[\psi\|y_0-x^*\|+\psi\|x_0-x^*\|]}{1-\psi_0\|x_0-x^*\|}\nonumber\\
&\times&\frac{1}{1-p(\|x_0-x^*\|)}\bigg(1+\frac{L_0}{2}\|z_0^{(1)}-x^*\|\bigg)\|z_0^{(1)}-x^*\|.
\end{eqnarray}}
Thus, we have
\begin{eqnarray}
\|z_0^{(2)}-x^*\|\le \eta_3(\|x_0-x^*\|)\|x_0-x^*\|<\rho.
\end{eqnarray}
Now, from the last sub-step of the method $(\ref{eqn:13})$, we have
{\small
\begin{eqnarray}\label{eqn:329}
&&x_1-x^*=z_0^{(2)}-x^*-(2[T'(y_0)]^{-1}-[T'(x_0)]^{-1})T(z_0^{(2)})\nonumber\\
&&=(z_0^{(2)}-x^*-[T'(z_0^{(2)})]^{-1}T'(z_0^{(2)}))+[T'(z_0^{(2)})]^{(-1)}T'(x^*)T'(x^*)^{-1}[T'(y_0)-T'(z_0^{(1)})]\nonumber\\
&&.[T'(y_0)]^{(-1)}T'(x^*)[T'(x^*)]^{-1}T(z_0^{(2)})+[T'(x_0)]^{(-1)}T'(x^*)[T'(x^*)]^{-1}[T'(y_0)-T'(x_0)]\nonumber\\
&&.[T'(y_0)]^{(-1)}T'(x^*)[T'(x^*)]^{-1}T(z_0^{(2)}).
\end{eqnarray}}
On expanding $T(z_0^{(2)})$ along $x^*$ and taking norm of the equation $(\ref{eqn:329})$, we get
{\tiny
\begin{eqnarray}
\|x_1-x^*\|&\le&\frac{1}{1-\psi_0\|z_0^{(2)}-x^*\|}\frac{\psi}{2}\|z_0^{(2)}-x^*\|.\|z_0^{(2)}-x^*\|+\frac{[\psi\|y_0-x^*\|+\psi\|z_0^{(2)}-x^*\|]}{1-\psi_0\|z_0^{(2)}-x^*\|}\nonumber\\
&\times&\|[T'(x^*)]^{-1}T(z_0^{(2)})\|+\frac{1}{1-\psi_0\|x_0-x^*\|}[\psi\|y_0-x^*\|+\psi\|x_0-x^*\|]\nonumber\\
&\times&\|[T'(x^*)]^{-1}T(z_0^{(2)})\|\nonumber\\
&\le&\frac{1}{1-\psi_0\|z_0^{(2)}-x^*\|}\frac{\psi}{2}\|z_0^{(2)}-x^*\|.\|z_0^{(2)}-x^*\|+\frac{[\psi\|y_0-x^*\|+\psi\|z_0^{(2)}-x^*\|]}{1-\psi_0\|z_0^{(2)}-x^*\|}\nonumber\\
&\times&\frac{1}{1-p(\|x_0-x^*\|)}\bigg(1+\frac{\psi_0}{2}\|z_0^{(2)}-x^*\|\bigg)\|z_0^{(2)}-x^*\|+\frac{[\psi\|y_0-x^*\|+\psi\|x_0-x^*\|]}{1-\psi_0\|x_0-x^*\|}\nonumber\\
&\times&\frac{1}{1-p(\|x_0-x^*\|)}\bigg(1+\frac{\psi_0}{2}\|z_0^{(2)}-x^*\|\bigg)\|z_0^{(2)}-x^*\|.
\end{eqnarray}}
Thus, we have
\begin{eqnarray}
\|x_1-x^*\|\le \eta_4(\|x_0-x^*\|)\|x_0-x^*\|<\rho,
\end{eqnarray}
which shows that for $n=0,\ x_1\in B(x^*,\rho).$ 
The function $H_4(a)=\eta_4(a)-1$ gives $H_4(0)<0$ and $H_4(\rho_3)>0$. Hence, $H_4(t)$ has at least one root in $(0,\rho_3)$. Let $\rho$ be the smallest root of $H_4(t)$ in $(0,\rho_3)$. Then, we have 
\begin{eqnarray}\label{eqn:330}
0<\rho<\rho_3<\rho_2<\rho_1<1/\psi_0.
\end{eqnarray}
and
\begin{equation}\label{eqn:331}
0<\eta_4(a)<1, for\ a\in [0,\rho).
\end{equation} 
By simply replacing $x_0, y_0, z_0^{(1)},z_0^{(2)}, x_1$ by $x_n, y_n, z_n^{(1)},z_n^{(2)}, x_{n+1}$ in the preceding estimates, we arrive at inequalities $(\ref{eqn:316})-(\ref{eqn:318})$. By the estimate, 
\begin{eqnarray}
\|x_{n+1}-x^*\|\le \eta_4(\|x_0-x^*\|)\|x_n-x^*\|<\rho.
\end{eqnarray}
We conclude that $\lim_{n\rightarrow\infty}x_n=x^*$ and $x_{n+1}\in B(x^*,\rho)$. Finally, to prove the uniqueness, let $y^*\in B(x^*,\rho)$ where $y^*\neq x^*$ with $T(y^*)=0$.
Define $F=\int_0^1T'(x^*+t(y^*-x^*))dt$. On expanding $T(y^*)$ along $x^*$ and using inequality $(\ref{eqn:312})$, we obtain
\begin{eqnarray}
\|[T'(x^*)]^{-1}\int_0^1[T'(x^*+t(y*-x^*)-T'(x^*)]dt\|\nonumber\\
\le\frac{\psi_0}{2}\|y^*-x^*\|\le\frac{\psi_0}{2}\varrho<1.
\end{eqnarray}
So, by Banach lemma, $\int_0^1[T'(x^*)]^{-1}[T'(x^*+t(y^*-x^*))]dt$  exists and invertible leading to the conclusion $x^*=y^*$, which completes the uniqueness part of the proof.
\end{proof}

					% Remark
%\textbf{Remark} $1:$ \cite{Argyros} If $\omega_0, \ \omega$ are constant functions, then $\omega_0(t)=L_0t$ and $\omega(t)=Lt$ satisfying $L_0<L$.

                       % Local convergence analysis of the method $(\ref{eqn:13})$ under H\"{o}lder condition
\section{\bf Local convergence analysis of the method $(\ref{eqn:13})$ under H\"{o}lder condition}
In this section, we move forward to present the local convergence analysis using H\"{o}lder condition because there are numerous nonlinear equations for which the assumptions based on Lipschitz condition fails to solve without using higher-order derivatives. As an illustration, we review the following problem \cite{Argyros5}.
\begin{equation}
T(x)(s)=x(s)-\int_0^1 G(s,t)\left(x(t)^\frac{5}{2}+\frac{x(t)^2}{2}\right)dt,
\end{equation}
where  $T:C[0,1]\rightarrow C[0,1]$ and the kernel $G$ is the Green's function defined on the interval $[0,1]\times[0,1]$ by
\begin{eqnarray}
G(s,t)=
\begin{cases}(1-s)t & t\le s, \\ s(1-t),& s\le t.\end{cases} \nonumber
\end{eqnarray}
Note that, 
\begin{eqnarray*}
\bigg\|\int_0^1G(s,t)\bigg\|\le \frac{1}{8}.
\end{eqnarray*}
Then,
\begin{equation}\label{eqn:42a}
\|G'(x)-G'(y))\|\le \frac{1}{8}\left(\frac{5}{2}\|x-y\|^\frac{3}{2}+\|x-y\|\right).
\end{equation}
Clearly, it can be seen that $G'$ does not satisfy the Lipschitz continuity condition. However, $G'$ is  H\"{o}lder continuous. For such kind of examples, we also derive the local convergence results. This analysis also generalizes the local convergence analysis presented in the previous section.
Supposing $q\in (0,1]$, $\kappa_0>0$ and $\kappa>0$ be two parameters with $\kappa_0\leq \kappa$, we define the functions  $\mu_1,\mu_2,\mu_3,\mu_4,p,M_1,M_2,M_3$ and $M_4$ on interval $\left[0,\left(\frac{1}{\kappa_0}\right)^\frac{1}{q}\right)$ by
{\tiny
\begin{eqnarray}\label{eqn:42}
\mu_4(a)&=&{\bigg[}\frac{\kappa\mu_3(a)^q.a^q}{(q+1)(1-\kappa_0\mu_3(a)^q.a^q)}\nonumber\\
&&+\frac{1}{1-\kappa_0\mu_3(a^q)a^q}\kappa[\mu_1(a)^q+\mu_3(a)^q]a^q\times
\left(\frac{1}{1-p(a)}\bigg[1+\frac{\kappa_0}{q+1}\mu_3(a)^q.a^q\bigg]\right)\nonumber\\
&&+\left(\frac{1}{1-\kappa_0a^q}.\frac{\kappa[1+\mu_1(a)^q]a^q}{1-p(a)}\times \bigg[1+\frac{\kappa_0}{q+1}\mu_3(a)^q.a^q\bigg]\right){\bigg]}\mu_3(a),
\end{eqnarray}}
where
\begin{eqnarray}\label{eqn:43}
\mu_1(a)=\frac{1}{1-\kappa_0a^q}{\bigg[}\frac{\kappa a^q}{q+1}+\frac{1+\frac{\kappa_0}{(q+1)}.a^q}{2}{\bigg]},
\end{eqnarray}
\begin{eqnarray}\label{eqn:44}
\mu_2(a)=\frac{1}{1-\kappa_0a^q}\bigg[\frac{\kappa a^q}{q+1}+\frac{\kappa[1+\mu_1(a)^q][\frac{\kappa_0}{(q+1)}.a^q+1].a^q}{1-p(a)}\bigg],
\end{eqnarray}
{\tiny
\begin{eqnarray}\label{eqn:45}
\mu_3(a)&=&{\bigg[}\frac{\kappa\mu_2(a)^q.a^q}{(q+1)(1-\kappa_0\mu_2(a)^q.a^q)}\nonumber\\
&&+\frac{1}{1-\kappa_0\mu_2(a^q)a^q}\kappa[\mu_1(a)^q+\mu_2(a)^q]a^q\times
\left(\frac{1}{1-p(a)}\bigg[1+\frac{\kappa_0}{q+1}\mu_2(a)^q.a^q\bigg]\right)\nonumber\\
&&+\left(\frac{1}{1-\kappa_0a^q}.\frac{\kappa[1+\mu_1(a)^q]a^q}{1-p(a)}\times \bigg[1+\frac{\kappa_0}{q+1}\mu_2(a)^q.a^q\bigg]\right){\bigg]}\mu_2,(a),
\end{eqnarray}}
and
\begin{eqnarray}\label{eqn:46}
p(a)=\kappa_0\mu_1(a)^q a^q.
\end{eqnarray}
Let
\begin{eqnarray}\label{eqn:47}
M_1(a)=\mu_1(a)-1,\ \ \ M_2(a)=\mu_2(a)-1,\\
M_3(a)=\mu_3(a)-1,\ \ \ 
M_4(a)=\mu_4(a)-1.
\end{eqnarray}
and 
\begin{equation}\label{eqn:41}
\rho_1=\left(\frac{q+1}{2\kappa+\kappa_0(3+2q)}\right)^\frac{1}{q}<\left(\frac{1}{\kappa_0}\right)^\frac{1}{q}.
\end{equation}
We observe that $\mu_1(\rho_1)=1$ and 
\begin{equation}\label{eqn:41a}
0\leq \mu_1(a)<1\ for\ a\in [0,\rho_1),\ \ \ \ 0<\rho_1<(1/\kappa_0)^\frac{1}{q}.
\end{equation}
%Now, we have that $M_1(0)=M_2(0)=M_3(0)=M_4(0)<0$ and $M_1(a)\rightarrow+\infty$ or a positive constant as $t\rightarrow \left(\left(\frac{1}{\kappa_0}\right)^\frac{1}{q}\right)^-$.The intermediate value theorem confirms the existence of the zeros of the function $M_2(a)$ in the interval $(0,\left(\frac{1}{\kappa_0}\right)^\frac{1}{q})$. 
We denote the smallest zero of $M_2(a)$ as $\rho_2$. Also, $M_2(\rho_1)>0$ and $\rho_1<\left(\frac{1}{\kappa_0}\right)^\frac{1}{q}$, which follows from that 
\begin{equation}\label{eqn:48}
0<\rho_2<\rho_1,\ \ \ 0<\mu_2(a)<1, for\ a\in [0,\rho_2).
\end{equation} 
Following this procedure, there comes the  existence of zeros of functions $M_i,\ i=1,2,3,4$ in the interval $(0,\rho_0)$. Denote by $\rho_i$, respectively, the smallest solution of functions $M_i's$. Define the radius of convergence $\rho$ and for each $a\in [0, \rho)$ by
\begin{eqnarray}\label{eqn:410}
\rho=min\{\rho_i\}, \ i=1,2,3,4, \ \ \ 
0\leq \mu_i(a)<1,
\end{eqnarray}
%and
 Also, we assume that there exists $x^*\in D$ for all $x,y\in D$ and fulfill the below conditions alongwith the assumption $(\ref{eqn:311})$:
\begin{eqnarray}\label{eqn:413}
\|[T'(x^*)]^{-1}(T'(x)-T'(x^*))\|&\le& \kappa_0\|x-x^*\|^q.
\end{eqnarray}
\begin{eqnarray}\label{eqn:414}
\|[T'(x^*)]^{-1}(T'(x)-T'(y))\|&\le& \kappa\|x-y\|^q,\\
B(x^*,\rho)&\subseteq& D,
\end{eqnarray}
             %Lemma
\begin{Lemma}\label{lm:42}
If $T$ satisfies the assumptions  $(\ref{eqn:413})$ and  $(\ref{eqn:414})$, consequently the inequalities given below hold for $x\in D$, $q=(0,1]$ and $t\in [0,1]$:
\begin{eqnarray}\label{eqn:421}
\|[T'(x^*)]^{-1}T'(x)\|&\le&1+ \kappa_0\|x-x^*\|^q.
\end{eqnarray}
\begin{eqnarray}\label{eqn:422}
\|[T'(x^*)]^{-1}(T'(x^*+t(x-x^*))\|&\le&1+ \kappa_0t^q\|x-x^*\|^q.
\end{eqnarray}
\begin{eqnarray}\label{eqn:423}
\|[T'(x^*)]^{-1}T(x)\|&\le&\left(1+ \frac{\kappa_0}{q+1}\|x-x^*\|^q\right).\|x-x^*\|.
\end{eqnarray}
\end{Lemma}                         
\begin{proof}
 By considering the hypothesis $(\ref{eqn:413})$, we get
\begin{eqnarray}\label{eqn:424}
\|[T'(x^*)]^{-1}T'(x)\|&\le&1+\|[T'(x^*)]^{-1}T'(x)-[T'(x^*)]^{-1}T'(x^*)\|\nonumber\\
&\le&1+ \kappa_0\|x-x^*\|^q.
\end{eqnarray}   
 In the similar manner, we can derive
 \begin{eqnarray}\label{eqn:425}
\|[T'(x^*)]^{-1}(T'(x^*+t(x-x^*))\|&\le& 1+\kappa_0\|x^*+t(x-x^*)-x^*\|^q\nonumber\\
&\le&1+ \kappa_0t^q\|x-x^*\|^q.
\end{eqnarray} 
Next,
\begin{eqnarray}\label{eqn:426}
\|[T'(x^*)]^{-1}T(x)&\le& \int_0^1[\kappa_0\|x^*+t(x-x^*)-x^*\|^q+1].\|x-x^*\|dt \nonumber\\
&\le&\left(1+ \frac{\kappa_0}{q+1}\|x-x^*\|^q\right).\|x-x^*\|.
\end{eqnarray} 
\end{proof}

                                                                                 %Theorem(2)
\begin{thm}\label{thm:41}
Suppose that there exists $x^*\in D$ and the Fr\'echet differentiable operator $T:D\subset X\rightarrow Y$ satisfies the assumptions  $(\ref{eqn:413})$ and  $(\ref{eqn:414})$. Then, the sequence $\{x_n\}$ generated by the method $(\ref{eqn:13})$ for $x_0\in B(x^*,\rho)\backslash\{x^*\}$ is well defined in $B(x^*,\rho)$ remains in $B(x^*,\rho)$ for each $n=0,1,2,\cdots$ and converges to $x^*$. Henceforward, the following measures holds for $n=0,1,2,\cdots$:
\begin{eqnarray}\label{eqn:415}
\|y_n-x^*\|\le \mu_1(\|x_n-x^*\|)\|x_n-x^*\|\le\|x_n-x^*\|<\rho,
\end{eqnarray}
\begin{eqnarray}\label{eqn:416}
\|z_n^{(1)}-x^*\|\le \mu_2(\|x_n-x^*\|)\|x_n-x^*\|\le\|x_n-x^*\|<\rho,
\end{eqnarray}
\begin{eqnarray}\label{eqn:417}
\|z_n^{(2)}-x^*\|\le \mu_3(\|x_n-x^*\|)\|x_n-x^*\|\le\|x_n-x^*\|<\rho,
\end{eqnarray}
and
\begin{eqnarray}\label{eqn:418}
\|x_{n+1}-x^*\|\le \mu_4(\|x_n-x^*\|)\|x_n-x^*\|\le\|x_n-x^*\|<\rho,
\end{eqnarray}
where the functions $\mu_i, i=1,2,3,4$ are defined by the expressions $(\ref{eqn:42})$ - $(\ref{eqn:45})$. Moreover, if there exists $\varrho\in\left [\rho,\left(\frac{1+q}{\kappa_0}\right)^\frac{1}{q}\right]$ such that $\overline {B(x^*,\varrho)}\subseteq D$. Then the limit point $x^*$ is the only solution of equation $T(x)=0$ in $D_1=D\cap \overline {B(x^*,\varrho)}$.
\end{thm}
                            %proof
\begin{proof}
Assume that $\|x_0-x^*\|^q<\frac{1}{\kappa_0}$ and using the hypotheses, $x_0\in B(x^*,\rho)\backslash\{x^*\}$ and inequality $(\ref{eqn:413})$, we have that
\begin{eqnarray}
\|[T'(x^*)]^{-1}(T'(x_0)-T'(x^*))\|\le \kappa_0\|x_0-x^*\|^q<1.
\end{eqnarray}
It follows from the above and the Banach lemma on invertible operator \cite{Argyros} that $[T'(x_0)]^{-1}\in L(Y,X)$
 or $T'(x_0)$ is invertible and
\begin{eqnarray}\label{eqn:t41a}
\|[T'(x_0)]^{-1}T'(x^*)\|\le\frac{1}{1-\kappa_0\|x_0-x^*\|^q}.
\end{eqnarray}
Now, $y_0$ is well defined by the first sub-step of the scheme $(\ref{eqn:13})$ and for $n=0$ we get the approximation $(\ref{eqn:321})$. 
%\begin{eqnarray}\label{eqn:t41}
%y_0-x^*&=&x_0-x^*-\frac{1}{2}[[T'(x_0)]^{-1}T(x_0)]\nonumber\\
%& =&\frac{1}{2}[[T'(x_0)]^{-1}T(x_0)]+[T'(x_0)]^{-1}[T'(x_0)(x_0-x^*)-T(x_0)+T(x^*)]||.\nonumber\\
%\end{eqnarray}
Expanding $T(x_0)$ along $x^*$ and taking the norm, we get the inequality $(\ref{eqn:322a})$. Finally, using the assumption $(\ref{eqn:413})$ and the inequality $(\ref{eqn:t41a})$, we get
%\begin{eqnarray}\label{eqn:t42}
%y_0-x^*&=&\frac{1}{2}[T'(x_0)]^{-1}\int_0^1T'(x^*+t(x_0-x^*))dt(x_0-x^*)\nonumber\\
%&+&[T'(x_0)]^{-1}\int_0^1[T'(x_0)-T'(x^*+t(x_0-x^*))]dt(x_0-x^*).\nonumber\\
%\end{eqnarray}
%On taking the norm of the equation $(\ref{eqn:t42})$, we get
%\begin{eqnarray}
%\|y_0-x^*\|&\le& \bigg\|\frac{1}{2}[T'(x_0)]^{-1}T'(x^*)\bigg\|\bigg[\int_0^1\|[T'(x^*]^{-1}[T'(x*+t(x_0-x^*)]dt\|\|x_0-x^*\|\bigg]\nonumber\\
%&&+\|[T'(x_0)]^{-1}T'(x^*)\|\int_0^1\|[T'(x^*)]^{-1}[T'(x_0)-T'(x^*+t(x_0-x^*))]dt\|\|x_0-x^*\|,\nonumber\\
%\end{eqnarray}
%Thus, we get
{\small
\begin{eqnarray}\label{eqn:t43}
\|y_0-x^*\|&\le& \frac{1}{1-\kappa_0\|x_0-x^*\|^q}\bigg[\frac{1}{2}{\bigg[}\frac{\kappa_0}{q+1}\|x_0-x^*\|^q+1{\bigg]}\nonumber\\
&+&\frac{\kappa}{q+1}\|x_0-x^*\|^q\bigg]\|x_0-x^*\|\nonumber\\
&\le&\mu_1(\|x_0-x^*\|)\|x_0-x^*\|<\rho.
\end{eqnarray}}
From the inequalities $(\ref{eqn:413})$ and $(\ref{eqn:t43})$, we have
\begin{eqnarray}
\|[T'(x^*)]^{-1}[T'(y_0)-T'(x^*)]\|&\le&\kappa_0\|y_0-x^*\|^q\nonumber\\
&\le&\kappa_0\mu_1(\|x_0-x^*\|)^q\|x_0-x^*\|^q\nonumber\\
&=&p(\|x_0-x^*\|)<1.
\end{eqnarray}
%Thus, by Banach lemma
%\begin{eqnarray}
%\|[T'(y_0)]^{-1}T'(x^*)\|\le\frac{1}{1-p(\|x_0-x^*\|)}.
%\end{eqnarray}
Again, from the second sub-step of the method $(\ref{eqn:13})$, we find the approximation $(\ref{eqn:326})$. 
%\begin{eqnarray}\label{eqn:t44}
%z_0^{(1)}-x^*&=&x_0-x^*-[T'(y_0)^{-1}T(x_0)]\nonumber\\
%&=&x_0-x^*-[T'(x_0)]^{-1}T(x_0)+[T'(x_0)]^{-1}[T'(y_0)-T'(x_0)]T'(y_0)^{-1}T(x_0).\nonumber\\
%\end{eqnarray}
On taking the norm, we get
{\tiny
\begin{eqnarray}
\|z_0^{(1)}-x^*\|&\le&\|x_0-x^*-[T'(x_0)]^{-1}T(x_0)\|+\|[T'(x_0)]^{-1}T'(x^*)\|.\nonumber\\
&&\|[T'(x^*)]^{-1}[T'(y_0)-T'(x_0)]\|\|[T'(y_0)]^{-1}T'(x^*)\|\|[T'(x^*)]^{-1}T(x_0)\|\nonumber\\
&\le&\frac{1}{1-\kappa_0\|x_0-x^*\|^q}\bigg[\frac{\kappa}{q+1}\|x_0-x^*\|^q\nonumber\\
&+&\frac{[\kappa\|y_0-x^*\|^q+\kappa(\|x_0-x^*\|^q)][\frac{\kappa_0}{q+1}\|x_0-x^*\|^q+1]}{1-p(\|x_0-x^*\|)}\bigg]\|x_0-x^*\|.
\end{eqnarray}}
Thus, we get
\begin{eqnarray}
\|z_0^{(1)}-x^*\|\le \mu_2(\|x_0-x^*\|)\|x_0-x^*\|\le\|x_0-x^*\|<\rho.
\end{eqnarray}
From the next sub-step of the method $(\ref{eqn:13})$, we have the approximation $(\ref{eqn:327})$. 
%\begin{eqnarray}\label{eqn:t45}
%&&z_0^{(2)}-x^*=z_0^{(1)}-x^*-(2[T'(y_0)]^{-1}-[T'(x_0)]^{-1})T(z_0^{(1)})\nonumber\\
%&&=(z_0^{(1)}-x^*-[T'(z_0^{(1)})]^{-1}T'(z_0^{(1)}))+[T'(z_0^{(1)})]^{(-1)}T'(x^*)T'(x^*)^{-1}[T'(y_0)-T'(z_0^{(1)})]\nonumber\\
%&&.[T'(y_0)]^{(-1)}T'(x^*)[T'(x^*)]^{-1}T(z_0^{(1)})+[T'(x_0)]^{(-1)}T'(x^*)[T'(x^*)]^{-1}[T'(y_0)-T'(x_0)]\nonumber\\
%&&.[T'(y_0)]^{(-1)}T'(x^*)[T'(x^*)]^{-1}T(z_0^{(1)}).
%\end{eqnarray}
On expanding $T(z_0^{(1)})$ along $x^*$ and taking norm, we get
{\tiny
\begin{eqnarray}
\|z_0^{(2)}-x^*\|&\le&\frac{1}{1-\kappa_0\|z_0^{(1)}-x^*\|^q}\frac{\kappa}{q+1}\|z_0^{(1)}-x^*\|^q.\|z_0^{(1)}-x^*\|+\frac{[\kappa\|y_0-x^*\|^q+\kappa\|z_0^{(1)}-x^*\|^q]}{1-\kappa_0\|z_0^{(1)}-x^*\|^q}\nonumber\\
&\times&\|[T'(x^*)]^{-1}T(z_0^{(1)})\|+\frac{1}{1-\kappa_0\|x_0-x^*\|^q}[\kappa\|y_0-x^*\|^q+\kappa\|x_0-x^*\|^q]\nonumber\\
&\times&\|[T'(x^*)]^{-1}T(z_0^{(1)})\|\nonumber\\
&\le&\frac{1}{1-\kappa_0\|z_0^{(1)}-x^*\|^q}\frac{\kappa}{q+1}\|z_0^{(1)}-x^*\|^q.\|z_0^{(1)}-x^*\|+\frac{[\kappa\|y_0-x^*\|^q+\kappa\|z_0^{(1)}-x^*\|^q]}{1-\kappa_0\|z_0^{(1)}-x^*\|^q}\nonumber\\
&\times&\frac{1}{1-p(\|x_0-x^*\|)}\bigg(1+\frac{\kappa_0}{q+1}\|z_0^{(1)}-x^*\|^q\bigg)\|z_0^{(1)}-x^*\|+\frac{[\kappa\|y_0-x^*\|^q+\kappa\|x_0-x^*\|^q]}{1-\kappa_0\|x_0-x^*\|^q}\nonumber\\
&\times&\frac{1}{1-p(\|x_0-x^*\|^q)}\bigg(1+\frac{\kappa_0}{q+1}\|z_0^{(1)}-x^*\|^q\bigg)\|z_0^{(1)}-x^*\|.
\end{eqnarray}}
Thus, we have
\begin{eqnarray}
\|z_0^{(2)}-x^*\|\le \kappa_3(\|x_0-x^*\|)\|x_0-x^*\|<\rho,
\end{eqnarray}
Finally, from the last sub-step of the method $(\ref{eqn:13})$ we have the approximation $(\ref{eqn:329})$. 
%\begin{eqnarray}\label{eqn:t46}
%x_1-x^*=z_0^{(2)}-x^*-(2[T'(y_0)]^{-1}-[T'(x_0)]^{-1})T(z_0^{(2)})\nonumber\\
%=(z_0^{(2)}-x^*-[T'(z_0^{(2)})]^{-1}T'(z_0^{(2)}))+[T'(z_0^{(2)})]^{(-1)}T'(x^*)T'(x^*)^{-1}[T'(y_0)-T'(z_0^{(1)})]\nonumber\\
%.[T'(y_0)]^{(-1)}T'(x^*)[T'(x^*)]^{-1}T(z_0^{(2)})+[T'(x_0)]^{(-1)}T'(x^*)[T'(x^*)]^{-1}[T'(y_0)-T'(x_0)]\nonumber\\
%.[T'(y_0)]^{(-1)}T'(x^*)[T'(x^*)]^{-1}T(z_0^{(2)})\nonumber\\
%\end{eqnarray}
On expanding $T(z_0^{(2)})$ along $x^*$ and taking norm, we get
{\tiny
\begin{eqnarray}
\|x_1-x^*\|&\le&\frac{1}{1-\kappa_0\|z_0^{(2)}-x^*\|^q}\frac{\kappa}{q+1}\|z_0^{(2)}-x^*\|^q.\|z_0^{(2)}-x^*\|+\frac{[\kappa\|y_0-x^*\|^q+\kappa\|z_0^{(2)}-x^*\|^q]}{1-\kappa_0\|z_0^{(2)}-x^*\|^q}\nonumber\\
&\times&\|[T'(x^*)]^{-1}T(z_0^{(2)})\|+\frac{1}{1-\kappa_0\|x_0-x^*\|^q}[\kappa\|y_0-x^*\|^q+\kappa\|x_0-x^*\|^q]\nonumber\\
&\times&\|[T'(x^*)]^{-1}T(z_0^{(2)})\|\nonumber\\
&\le&\frac{1}{1-\kappa_0\|z_0^{(2)}-x^*\|^q}\frac{\kappa}{q+1}\|z_0^{(2)}-x^*\|^q.\|z_0^{(2)}-x^*\|+\frac{[\kappa\|y_0-x^*\|^q+\kappa\|z_0^{(2)}-x^*\|^q]}{1-\kappa_0\|z_0^{(2)}-x^*\|^q}\nonumber\\
&\times&\frac{1}{1-p(\|x_0-x^*\|)}\bigg(1+\frac{\kappa_0}{q+1}\|z_0^{(2)}-x^*\|^q\bigg)\|z_0^{(2)}-x^*\|+\frac{[\kappa\|y_0-x^*\|^q+\kappa\|x_0-x^*\|^q]}{1-\kappa_0\|x_0-x^*\|^q}\nonumber\\
&\times&\frac{1}{1-p(\|x_0-x^*\|)}\bigg(1+\frac{\kappa_0}{q+1}\|z_0^{(2)}-x^*\|^q\bigg)\|z_0^{(2)}-x^*\|.
\end{eqnarray}}
Thus, we have
\begin{eqnarray}
\|x_1-x^*\|\le \mu_4(\|x_0-x^*\|)\|x_0-x^*\|<\rho,
\end{eqnarray}
which shows that for $n=0,\ x_1\in B(x^*,\rho).$ 
The function $M_4(a)=\mu_4(a)-1$ gives $M_4(0)<0$ and $M_4(\rho_3)>0$. Hence, $M_4(t)$ has at least one root in $(0,\rho_3)$. Let $\rho$ be the smallest root of $M_4(t)$ in $(0,\rho_3)$. Then, we have 
\begin{eqnarray}\label{eqn:t47}
0<\rho<\rho_3<\rho_2<\rho_1<\left(\frac{1}{\kappa_0}\right)^\frac{1}{q}.
\end{eqnarray}
and
\begin{equation}\label{eqn:t48}
0<\mu_4(a)<1, for\ a\in [0,\rho).
\end{equation} 
By simply replacing $x_0, y_0, z_0^{(1)},z_0^{(2)}, x_1$ by $x_n, y_n, z_n^{(1)},z_n^{(2)}, x_{n+1}$ in the preceding estimates, we arrive at inequalities $(\ref{eqn:415})-(\ref{eqn:418})$. By the estimate 
\begin{eqnarray}
\|x_{n+1}-x^*\|\le \mu_4(\|x_0-x^*\|)\|x_n-x^*\|<\rho.
\end{eqnarray}
We conclude that $\lim_{n\rightarrow\infty}x_n=x^*$ and $x_{n+1}\in B(x^*,\rho)$. Finally, to prove the uniqueness, let $y^*\in B(x^*,\rho)$ where $y^*\neq x^*$ with $T(y^*)=0$.
Define $F=\int_0^1T'(x^*+t(y^*-x^*))dt$. On expanding $T(y^*)$ along $x^*$ and using inequality $(\ref{eqn:413})$, we obtain
\begin{eqnarray}
\|[T'(x^*)]^{-1}\int_0^1[T'(x^*+t(y*-x^*)-T'(x^*)]dt\|\nonumber\\
\le\frac{\kappa_0}{q+1}\|y^*-x^*\|^q\le\frac{\kappa_0}{q+1}\varrho^q<1.
\end{eqnarray}
So, by Banach lemma, $\int_0^1[T'(x^*)]^{-1}[T'(x^*+t(y^*-x^*))]dt$  exists and invertible leading to the conclusion $x^*=y^*$, which completes the uniqueness part of the proof.
\end{proof}

						%Numerical Example%

						%Example (1)

\section{\bf Applications with numerical examples }
In this section, two numerical examples are worked out to demonstrate the efficiency of our local convergence analysis by giving the radii of convergence for the scheme $(\ref{eqn:13})$. We obtain better results using our technique.\\\\
\textbf{Example 4.1}\cite{Bhavna} Returning back to the illustration example given in the introduction of this study of Lipschitz, 
The unique solution is $x^*=1$. The consecutive derivatives of $f$ are
\begin{eqnarray*}
f'(x)&=&3x^2logx^2+5x^4-4x^3+2x^2,\\
f''(x)&=&6xlogx^2+20x^3-12x^2+10x,\\
f'''(x)&=&6logx^2+60x^2-24x+22.
\end{eqnarray*}
It can be easily visible that $f'''$ is unbounded on $D$. Nevertheless, all the assumptions of the Theorem $(\ref{thm:31})$ for the iterative method $(\ref{eqn:13})$ are satisfied and hence applying the convergence results with $x^*=1$, we obtain $\psi_0=\psi=96.6628$. Here, we will use the iterative method $(\ref{eqn:13})$ and compare it with the scheme given by Cordero et al. \cite{Cordero1} and we denote it by CHMT. From the above Table $(\ref{tab:1})$, we can see that the radius $\rho$ of convergence computed for the method $(\ref{eqn:13})$ seems to be finer than the method CHMT which local convergence was discussed in the article \cite{Liu}. As a result, the approach being evaluated is more powerful.
\begin{table}
\centering
\caption{Comparison of convergence radius (Example 4.1)}\label{tab:1}
\begin{tabular}{|c| c| c|} \hline
$Radius$      & Method $(\ref{eqn:13})$    & CHMT$(\theta=-2)$ \\ \hline
$\rho_1$          & 0.00295578                  & 0.006689      \\ 
$\rho_2$          &  0.00246894                  &  0.005750     \\ 
$\rho_3$              &0.00217353                &0.003001     \\
$\rho_4$             &0.00208131                   &0.001943     \\
$\rho$                & 0.00208131                  & 0.001943     \\
\hline
\end{tabular}
\end{table}

\textbf {Example 4.2}\cite{Legaz} In order to show the applicability of the results presented in this paper, we consider the following Planck's radiation law problem \cite{Jain}:
\begin{eqnarray*}
\phi(\lambda)=\frac{8\pi c P\lambda^{-5}}{e^{\frac{c P}{\lambda BT}}-1},
\end{eqnarray*}
which calculates the energy density within an isothermal blackbody, where 
\begin{itemize}
\item $\lambda$ is the wavelength of the radiation
\item $T$ is the absolute temperature of the blackbody
\item $B$ is Boltzmann's constant 
\item $P$ is the Planck's constant
\item $c$ is the speed of light.
\end{itemize}
Suppose, we would like to determine wavelength $\lambda$ which corresponds to maximum energy density $\phi(\lambda)$. Therefore, solving for maxima we define 
\begin{eqnarray*}
f(x)=e^{-x}-1+\frac{x}{5}.
\end{eqnarray*}
We have $x^*=4.965114$ and $F'(x^*)=0.193023$. Then, on using assumptions $(\ref{eqn:413})-(\ref{eqn:414})$ we have that, $q=1$, $\kappa_0=0.0608658<\kappa=0.094888.$ Table $(\ref{tab:5})$ displays the radius $\rho$ of convergence by the discussed method $(\ref{eqn:13})$ along with the existing scheme KFS analyzed in the reference \cite{Legaz}. We discovered that when compared the provided method enlarges the radius of the convergence ball.
\begin{table}
\centering
\caption{Comparison of convergence radius (Example 4.2)}\label{tab:5}
\begin{tabular}{|c| c| c|} \hline
$Radius$      & Method $(\ref{eqn:13})$     & KFS \\ \hline
$\rho_1$       & 4.04772                       & 9.23282       \\ 
$\rho_2$       &2.99797                       &2.40532 \\ 
$\rho_3$        & 2.58569                     & 1.11454       \\
$\rho_4$        &2.45972                       & \\
$\rho$            &2.45972                       &1.11454        \\
\hline
\end{tabular}
\end{table}
%Thus, we guarantee the convergence of the method $(\ref{eqn:13})$ with radius 
%\begin{eqnarray}
%\rho=2.45972
%\end{eqnarray}
                                                                      %Example 6

%
\textbf {Example 4.3}\cite{Argyros5} Returning back to the illustration example given in the introduction of this study of H\"{o}lder, in view of $(\ref{eqn:42a})$, the previous results, which required $G^{(4)}$ to be bounded, are no longer valid. However, our findings are applicable and therefore our study, on the other hand, is valuable in ensuring that these techniques for this problem are converging.
Here, $q=1$ and $\kappa_0=\kappa=\frac{1}{8}\left(\frac{5}{2}\sqrt{2}+1\right)$. Indeed, using the above choices of functions $\kappa,\ \kappa_0$ and expression $(\ref{eqn:41})$,  the radius $\rho$ of convergence is computed in Table $(\ref{tab:6})$. 
\begin{table}
\centering
\caption{Convergence radius for (Example 4.3)}\label{tab:6}
\begin{tabular}{|c| c| } \hline
$Radius$      & Method $(\ref{eqn:13})$\\ \hline
$\rho_1$          & 0.503957                        \\ 
$\rho_2$          &0.420951              \\ 
$\rho_3$              & 0.378541                   \\
$\rho_4$             &0.363397      \\
$\rho$                &0.363397                \\
\hline
\end{tabular}
\end{table}
Thus, we guarantee the convergence of the method $(\ref{eqn:13})$ with radius $\rho=0.363397$.
%\begin{eqnarray}
%\end{eqnarray}
%From the above table, we can see that our work not only improves the application of these schemes, but also provides the methods' convergence radii.

					%Conclusion

\section{Conclusions}
In the presented work, we have analyzed the local convergence analysis of the efficient seventh order method for solving the nonlinear equation in Banach spaces. A convergence theorem for existence and uniqueness of the solution has been established followed by its error bounds giving the benefit that the iterative method always converges to the solution. This local convergence analysis is applicable in solving such problems for which higher-order derivative based previous studies fail. Later, we relaxed these assumptions and derive convergence results under H\"{o}lder condition for solving different types of nonlinear integral equations, which are not solvable by the previous approach. 
%In the future work, we will try to discuss the local convergence results under more relaxed and weaker conditions.

Akanksha Saxena\\
Department of Mathematics\\
Maulana Azad National Institute of Technology\\
 Bhopal, M.P. India-462003.\\
Email: akanksha.sai121@gmail.com.\\\\
J. P. Jaiswal\\
Department of Mathematics\\
Guru Ghasidas Vishwavidyalaya ( A Central University)\\
Bilaspur, C.G. India-495009.\\
Email: asstprofjpmanit@gmail.com.\\\\
K. R. Pardasani\\
Department of Mathematics\\
Maulana Azad National Institute of Technology\\
 Bhopal, M.P. India-462003.\\
Email: kamalrajp@rediffmail.com.\\\\

\end{document}